\documentclass[12pt]{article}
\usepackage{latexsym, amssymb, amsmath, amscd, amsfonts, epsfig, graphicx, colordvi,verbatim,ifpdf}
\usepackage{amsfonts, amsmath, amssymb}
\usepackage{amssymb,amsfonts,amsmath,latexsym,epsfig,cite, psfrag,eepic,color}
\usepackage{amscd,graphics}
\usepackage{latexsym, amssymb,  amsmath,amscd, amsfonts, epsfig, graphicx, colordvi,amsthm}

\usepackage{graphicx}
\usepackage{color}
\usepackage{ifpdf}
\usepackage{fancybox}

\usepackage{float}

\newtheorem{thm}{Theorem}[section]

\newtheorem{rem}[thm]{\it Remark}

\newtheorem{lem}[thm]{Lemma}

\newtheorem{exa}[thm]{Example}
\def\pf{\noindent{\it Proof.} }
\def\qed{\nopagebreak\hfill{\rule{4pt}{7pt}}
\medbreak}

\numberwithin{equation}{section}

\def\qed{\nopagebreak\hfill{\rule{4pt}{7pt}}
\medbreak}

\newlength{\boxedparwidth}
  {\begin{center} \begin{tabular}{|@{\hspace{.315in}}c@{\hspace{.15in}}|}
                  \hline \\ \begin{minipage}[t]{\boxedparwidth}
                  \setlength{\parindent}{.25in}}%
  {\end{minipage} \\ \\ \hline \end{tabular} \end{center}}

\allowdisplaybreaks
\begin{document}

\begin{center}

 {\large \bf 
Gap between the largest and smallest parts of partitions and Berkovich and Uncu's conjectures }
\end{center}

\begin{center}
{Wenston J.T. Zang}$^{1}$ and
  {Jiang Zeng}$^{2}$ \vskip 2mm
$^{1}$Institute for  Advanced Study in Mathematics\\[2pt]
   Harbin Institute of Technology, Heilongjiang 150001, P.R. China\\[8pt]

   \vskip 2mm

 $^{2}$Univ Lyon, Universit\'e Claude Bernard Lyon 1, CNRS UMR 5208,
 Institut Camille \\[2pt]
 Jordan, 43 blvd. du 11 novembre 1918, F-69622 Villeurbanne cedex,
 France\\[5pt]
  $^1$zang@hit.edu.cn, \quad $^2$zeng@math.univ-lyon1.fr
\end{center}
\vskip 6mm \noindent {\bf Abstract.}
We prove three main  conjectures of Berkovich and Uncu  (Ann. Comb. 23 (2019) 263--284) on the inequalities
between  the numbers of  partitions of $n$ with bounded gap between largest and smallest parts for sufficiently large $n$.  Actually our theorems are stronger than their original conjectures.
The analytic version of
our results shows that the coefficients of
 some  partition $q$-series are eventually positive.



\section{Introduction}
Let $n$ be a positive integer, a \emph{partition} of $n$ is a nonincreasing finite sequence of positive integers $\lambda_1, \lambda_2, \ldots , \lambda_k$ whose sum is $n$. Each $\lambda_i$
is called a \emph{part}
of the partition.
In \cite{Ber-Unc-19}, Berkovich and Uncu proved various inequalities   between
the numbers of partitions with the  bound on
the  largest part and some restrictions on occurrences of parts, and also make several conjectures. To be specific, we introduce
the following definitions:
\begin{enumerate}
\item Let ${\mathcal C}_{L,s,2}(n)$ (resp. $c_{L,s,2}(n)$)
be the set (resp. number) of partitions of $n$ with parts
in the domain $\{s+1, \ldots, s+L\}$.
\item Let ${\mathcal F}_{L,s,k}(n)$ (resp. $f_{L,s,k}(n)$) denote the set (resp. number) of partitions of $n$
 with  the smallest part $s$, the largest part at most $L+s$,
 and no part equal to $k$.
\end{enumerate}

In this paper, motivated by the open problems Conjecture 3.2, Conjecture 3.3 and Conjecture 7.1 in \cite{Ber-Unc-19}, we shall prove the following two main theorems.
\begin{thm}\label{thm-main}
For integer $s\geq 1$, $L\geq 3$ and $s+L\geq k\geq \max\{s+1,L\}$, there exists an integer $M$ which only depends on $s$ such that for $n\ge M$,
\begin{equation}
f_{L,s,k}(n)\geq c_{L,s,2}(n).
\end{equation}
\end{thm}

\begin{rem}
Berkovich and Uncu~\cite[Theorems 1.1 and  3.1]{Ber-Unc-19}
proved Theorem~\ref{thm-main}  for  $s=1$ (resp. $s=2$), $k=L$ (resp. $k=L+1$)  with   $M=1$ (resp. $M=10$).
They also conjectured
the cases $k=s+L-1$ and $k=L$ of Theorem~\ref{thm-main}~~\cite[Conjectures 3.2 and  3.3]{Ber-Unc-19}.
\end{rem}

By  the elementary theory of partitions~\cite[Chapters 1--3]{And-1998} it is not difficult to see that
the generating functions of  $c_{L,s,2}(n)$'s and $f_{L,s,1}(n)$'s  read as follows:
\begin{align}
\sum_{n=0}^\infty c_{L,s,2}(n) q^n&=\frac{1}{(q^{s+1};q)_{L}},\label{gf1}\\
\sum_{n=1}^\infty f_{L,s,k}(n)q^n&=\frac{q^s(1-q^k)}{(q^s;q)_{L+1}}.\label{gf2}
\end{align}
Here, we use the standard  $q$-notation~\cite{And-1998}:
\begin{align*}
(a;q)_n=(1-a)(1-aq)\ldots (1-aq^{n-1}).
\end{align*}

Recall that  a series $\sum_{n\geq 0}a_nq^n\in {\mathbb R}[[q]]$ is called
 {\it eventually positive} if there exists an integer $M\geq 0$ such that $a_n>0$ for all $n>M$.
For instance, Theorem~\ref{thm-main} and \eqref{gf1} and \eqref{gf2}  imply that
 the $q$-series
\[ \frac{q^s(1-q^k)-(1-q^s)}{(q^s;q)_{L+1}}\] is eventually positive.
In general, we derive the following theorem.

\begin{thm}\label{thm-main0-2}
For integers $L\ge 3$, $s\geq  1$, $r\geq 0$ and $k_1>k_2\geq 1$, the series
\begin{align}\label{def:H}
H^*_{L,s,r,k_1,k_2}(q): =\frac{q^r(1-q^{k_1})-(1-q^{k_2})}{(q^s;q)_{L+1}}
\end{align}
 is eventually positive.
\end{thm}

\begin{rem}
Set $r=s$, $k_1=k$ and $k_2=s$, we confirm the conjecture raised by Berkovich and Uncu~\cite[Conjecture~7.1]{Ber-Unc-19}.
\end{rem}

Let ${\mathcal A}=\{a_1,a_2,\ldots,a_m\}$ be a set of $m$ positive integers.
Denote by $p_{\mathcal A}(n)$ the number of nonnegative integer
solutions of the  diophantine equation
$a_1 x_1+\cdots +a_mx_m=n$, i.e.,
\begin{equation}
\sum_{n=0}^\infty p_{\mathcal A}(n)q^n=\frac{1}{(1-q^{a_1})\ldots (1-q^{a_m})}.
\end{equation}
It should be noted that $p_{\mathcal A}(n)$ is closely related to the
\emph{Frobeinus coin problem},
 see \cite{Alf-2000,Sel-1977} or  https://en.wikipedia.org/wiki/Coin\_problem
 for more details. We shall need
the following result,  see \cite{Be-Ge-Ko-01} or \cite[Theorem 3.15.2]{Wilf-1994} for an elementary proof.
\begin{thm}[Frobeinus-Schur]\label{schur-thm}
If  $\gcd(a_1,\ldots,a_m)=1$, then
\begin{equation}
p_{\mathcal A}(n)\sim \frac{n^{m-1}}{(m-1)!a_1a_2\cdots a_m}.
\end{equation}
\end{thm}

This paper is organized as follows. In Section 2, we first give two weak forms of Theorem \ref{thm-main}, we then prove Theorem \ref{thm-main} with the aid of  these two weak forms. In Section 3, we give a proof of Theorem \ref{thm-main0-2}.

\section{Proof of Theorem \ref{thm-main}}

In this section, we give a proof of Theorem \ref{thm-main}. To this end, we first show the following two theorems, namely Theorems \ref{thm-asy} and \ref{thm-main-inf}, which can be view as weak forms of Theorem\ref{thm-main}. We then prove Theorem \ref{thm-main} with the aid of Theorems \ref{thm-asy} and \ref{thm-main-inf}.

\begin{thm}\label{thm-asy}
Given integer $s$ and $L\geq 3$, there exists $M_{L,s}$ depending
 on $L$ and $s$ such that for any $\max\{s+1, L\}\leq k\leq s+L$ and $n\geq M_{L,s}$,
\begin{equation}\label{equ-flskn-geq-cls3}
f_{L,s,k}(n)\geq c_{L,s,2}(n).
\end{equation}
\end{thm}

\pf By definition, we see that $f_{L,s,k}(n)$ is the number of nonnegative integer solutions of the equation $sx_s+(s+1)x_{s+1}+\cdots+(s+L)x_{s+L}=n$, where $x_s\geq 1$ and $x_{k}=0$. Let $A:=\{s,s+1,\ldots,s+L\}$, from the definition of $p_A(n)$, we deduce that
\begin{equation}
f_{L,s,k}(n)=p_{A\setminus \{k\}}(n-s).
\end{equation}
Similarly, by the definition of $c_{L,s,2}(n)$,
\begin{equation}
c_{L,s,2}(n)=p_{A\setminus\{s\}}(n).
\end{equation}
As $L\geq 3$ both  $A\setminus \{k\}$ and $A\setminus \{s\}$ contain two consecutive integers, thus
$\gcd(A\setminus \{k\})=\gcd(A\setminus \{s\})=1$. Hence by Theorem \ref{schur-thm},
\[f_{L,s,k}(n)\sim \frac{k(n-s)^{L}}{L!s(s+1)\cdots (s+L)}\]
and
\[c_{L,s,2}(n)\sim \frac{s(n-s)^{L}}{L!s(s+1)\cdots (s+L)}.\]
Therefore,
\begin{equation}\label{equ-sim-flsk}
f_{L,s,k}(n)-c_{L,s,2}(n)\sim \frac{(k-s)(n-s)^{L}}{L!s(s+1)\cdots (s+L)}.
\end{equation}
From \eqref{equ-sim-flsk}, we see that there exists $M_{L,s,k}$ such that for $n\geq M_{L,s,k}$, \eqref{equ-flskn-geq-cls3} holds. When $L\ge s+1$, set
\[M_{L,s}:=\max\{M_{L,s,L},M_{L,s,L+1},\ldots M_{L,s,s+L}\};\]
and when $L\le s$, set
\[M_{L,s}:=\max\{M_{L,s,s+1},M_{L,s,s+2},\ldots M_{L,s,s+L}\}.\]
Clearly, for $n\ge M_{L,s}$, \eqref{equ-flskn-geq-cls3} valid and $M_{L,s}$ only depends on $s$ and $L$. This completes the proof.\qed

We next give another weak form of Theorem \ref{thm-main}.

\begin{thm}\label{thm-main-inf}
Let  $L$,  $s$ and $k$ be positive integers such that $L\geq 2s^3+5s^2+1$ and  $L\leq k\leq s+L$. Then, for any $n\ge 2s^5+8s^4+s^3-14s^2+3s+1$, we have
\[f_{L,s,k}(n)\geq c_{L,s,2}(n).\]
\end{thm}

To prove Theorem \ref{thm-main-inf}, we shall build an injection $\phi: C_{L,s,2}(n)\to  F_{L,s,k}(n)$.
More specifically, we shall divide $\phi$ into
five injections $\phi_i: C^i_{L,s,2}(n)\to F^i_{L,s,k}(n)$ for $1\leq i\leq 5$, where $\{C^1_{L,s,2}(n), \ldots,  C^5_{L,s,2}(n)\}$ is a set partition of
$ C_{L,s,2}(n)$, and
$(F^1_{L,s,k}(n), \ldots, F^5_{L,s,k}(n))$  is a sequence of
  five disjoint subsets of $F_{L,s,k}(n)$.

We denote each partition
$\alpha\in C_{L,s,2}(n)$ by
$\alpha=((s+1)^{f_{s+1}}\ldots (s+L)^{f_{s+L}})$, where $f_i$ is the number of occurrences of $i$  in $\alpha$. The five  subsets  $C^i_{L,s,2}(n)$ are defined as follows.
\begin{itemize}
\item[(1)] $C^1_{L,s,2}(n)$ is the set of partitions in $C_{L,s,2}(n)$ such that $f_{k}=0$ and there exists $a\geq 2$ such that $f_{as}\geq 1$.
\item[(2)] $C^2_{L,s,2}(n)$ is the set of partitions in $C_{L,s,2}(n)$ such that $f_{k}=0$ and  $f_{as}=0$ for all $a\geq 2$. Moreover, there exists an integer $j$ such that $2s^2+5s-1\ge j\ge s+1$ and $f_j\geq s$.
\item[(3)] $C^3_{L,s,2}(n)$ is the set of partitions in $C_{L,s,2}(n)$ such that $f_{k}=0$ and  $f_{as}=0$ for all $a\geq 2$. Moreover, for any $j$ such that $2s^2+5s-1\ge j\ge s+1$ we have $f_j\leq s-1$.
\item[(4)] $C^4_{L,s,2}(n)$ is the set of partitions in $C_{L,s,2}(n)$ such that $f_{k}\geq 2$.
\item[(5)] $C^5_{L,s,2}(n)$ is the set of partitions in $C_{L,s,2}(n)$ such that $f_{k}=1$.
\end{itemize}
Similarly,  we denote each partition $\beta\in F_{L,s,k}(n)$ by $\beta=(s^{g_s}(s+1)^{g_{s+1}}\ldots (s+L)^{g_{s+L}})$, where $g_i$ is the number of occurrences of
$i$ in $\beta$. From the definition of $F_{L,s,k}(n)$, we see that $g_s\geq 1$ and $g_k=0$. Writing $k=rs+t$ with $0\leq t\leq s-1$, we define the  five
 subsets  $F^i_{L,s,2}(n)$  as follows.
\begin{itemize}
\item[(1)] $F^1_{L,s,k}(n)$ is the set of partitions in $F_{L,s,k}(n)$ such that $r+1\geq g_s\geq 2$ and for any $2\leq i< g_s$, $g_{is}=0$.
\item[(2)] $F^2_{L,s,k}(n)$ is the set of partitions in $F_{L,s,k}(n)$ such that $g_{s}=1$ and  there exists $i\geq 2$ such that $g_{is}=1$. Moreover, for any $j\neq 1,i$, we have $g_{js}=0$.
\item[(3)] $F^3_{L,s,k}(n)$ is the set of partitions in $F_{L,s,k}(n)$ such that $g_{s}=1$ and  $g_{2s}+g_{3s}\geq 2$.
\item[(4)] $F^4_{L,s,k}(n)$ is the set of partitions in $F_{L,s,k}(n)$ such that $g_{s}\geq 2r-4$.
\item[(5)] $F^5_{L,s,k}(n)$ is the set of partitions in $F_{L,s,k}(n)$ such that $g_{s}=r-4$ and $g_{2s}\geq 1$.
\end{itemize}
Since $k\geq L\geq 2s^3+5s^2+1$, we derive that $r\geq 2s^2+5s\geq 7$.
Therefore,  $2r-4>r+1$, which implies that $F^1_{L,s,k}(n)\cap F^4_{L,s,k}(n)=\emptyset$; also $r-4\geq 3$, which implies that $F^1_{L,s,k}(n)\cap F^5_{L,s,k}(n)=\emptyset$, $F^2_{L,s,k}(n)\cap F^5_{L,s,k}(n)=\emptyset$ and $F^3_{L,s,k}(n)\cap F^5_{L,s,k}(n)=\emptyset$.

Now we proceed to construct the five injections explicitly.

\begin{lem}\label{lem-inj-1}
There is an explicit injection $\phi_1:C^1_{L,s,2}(n)\to F^1_{L,s,k}(n)$.
\end{lem}

\begin{proof} Let  $\alpha=((s+1)^{f_{s+1}}\ldots (s+L)^{f_{s+L}})\in C^1_{L,s,2}(n)$ with  $f_k=0$.  Let $a\geq 2$ be the smallest integer such that
 $f_{as}\geq 1$. We  define
\begin{equation}\label{equ-phi1}
\phi_1(\alpha):=(s^{g_s}\ldots (s+L)^{g_{s+L}})=(s^{a}\ldots (as)^{f_{as}-1}\ldots (s+L)^{f_{s+L}}).
\end{equation}
Clearly, $|\phi_1(\alpha)|=|\alpha|=n$, $g_k=f_k=0$ and $g_s=a\geq 2$. Moreover from $L\leq k\leq s+L$ and $k=rs+t$, we deduce that $as\leq s+L\leq s+k=(r+1)s+t$. Thus $a\leq r+1$. Hence $r+1\geq a=g_s\geq 2$. From the choice of $a$, we see that for any $2\leq i< a$,  we have $g_{is}=f_{is}=0$. From the above analysis, we derive that $\phi_1(\alpha)\in F^1_{L,s,k}(n)$.

It remains to show that $\phi_1$ is an injection. Let
$$I^1_{L,s,k}(n)=\{\phi_1(\alpha)\colon \alpha\in C^1_{L,s,2}(n)\}$$
 be the image set of $\phi_1$, which has been shown to be a subset of $F^1_{L,s,k}(n)$. We wish to construct a map $\psi_1\colon I^1_{L,s,k}(n)\rightarrow C^1_{L,s,2}(n)$ such that for any $\alpha\in C^1_{L,s,2}(n)$,
 \[\psi_1(\phi_1(\alpha))=\alpha.\]
 Let $\beta=(s^{g_s}\ldots (s+L)^{g_{s+L}})\in I^1_{L,s,k}(n)$, that is, there exists $\alpha\in C^1_{L,s,2}(n)$ such that $\phi_1(\alpha)=\beta$. From the construction \eqref{equ-phi1}, we see that $sg_s\leq s+L$ and $sg_s\neq k$. Define
 \[\psi_1(\beta)=((s+1)^{g_{s+1}}\ldots (sg_s)^{g_{sg_s}+1}\ldots (s+L)^{g_{s+L}}).\]
It is easy to check that $\psi_1(\beta)\in C^1_{L,s,2}(n)$ and $\psi_1(\phi_1(\alpha))=\alpha$. This completes the proof.
\end{proof}

\begin{exa}
For $s=3$, $L=110$, $k=112$ and $n=1205$, let
$$\alpha:=(9^7,15^3,16^2,20^9,30^8,40^2,80,97^5).$$
It is trivial to check that $\alpha\in C^1_{110,3,2}(1205)$. Applying $\phi_1$ to $\alpha$, we see that $a=3$. Hence
\[\phi_1(\alpha)=(3^3,9^6,15^3,16^2,20^9,30^8,40^2,80,97^5),\]
which is  in $F^1_{110,3,112}(1205)$. Moreover, applying $\psi_1$ to $\phi_1(\alpha)$, we recover $\alpha$.
\end{exa}

\begin{lem}\label{lem-inj-2}
There is an explicit injection $\phi_2: C^2_{L,s,2}(n)\to F^2_{L,s,k}(n)$.
\end{lem}

\pf For $\alpha=((s+1)^{f_{s+1}}\ldots (s+L)^{f_{s+L}})\in C^2_{L,s,2}(n)$, by definition we see $f_k=0$, for any $a\geq 2$ we have $f_{as}=0$. Moreover, there exists $s+1\leq j\leq 2s^2+5s-1$ such that $f_j\geq s$. We choose such $j$ to be minimum, that is, $j=\min\{i\colon f_{i}\geq s\}$. Define
\begin{equation}\label{equ-phi2}
\phi_2(\alpha)=(s^{g_s}\ldots (s+L)^{g_{s+L}})=(s^{1}\ldots (j)^{f_{j}-s}\ldots ((j-1)s)^1\ldots (s+L)^{f_{s+L}}).
\end{equation}

From $k\geq L\geq 2s^3+5s^2+1$ and $j\leq 2s^2+5s-1$, we deduce that $k>s(2s^2+5s)>(j-1)s$. Thus $f_k=g_k=0$. Moreover, it is clear to see that $g_s=g_{(j-1)s}=1$ and for any $i\neq 1, j-1$, $g_{is}=f_{is}=0$. Furthermore, $|\phi_2(\alpha)|=|\alpha|-js+s+(j-1)s=n$. Hence we derive that $\phi_2(\alpha)\in F^2_{L,s,k}(n)$.

It remains to show that  $\phi_2$ is an injection. Let
$$I^2_{L,s,k}(n)=\{\phi_2(\alpha)\colon \alpha\in C^2_{L,s,2}(n)\}$$
 be the image set of $\phi_2$, which has been shown to be a subset of $F^2_{L,s,k}(n)$. We wish to construct a map $\psi_2\colon I^2_{L,s,k}(n)\rightarrow C^2_{L,s,2}(n)$ such that for any $\alpha\in C^2_{L,s,2}(n)$,
 \[\psi_2(\phi_2(\alpha))=\alpha.\]
 Let $\beta=(s^{g_s}\ldots (s+L)^{g_{s+L}})\in I^2_{L,s,k}(n)$, that is, there exists $\alpha\in C^2_{L,s,2}(n)$ such that $\phi_2(\alpha)=\beta$. From the definition of $F^2_{L,s,k}(n)$, we see that there exists a unique $i\ge 2$ such that $g_{is}=1$. By \eqref{equ-phi2}, we have $k\geq L>2s^2+5s-1\geq i+1\geq s+1$. Moreover, since $j$ is not a multiple of $s$, we see that $i+1$ is not a multiple of $s$. Hence we may  define
 \[\psi_2(\beta)=((s+1)^{g_{s+1}}\ldots (i+1)^{g_{i+1}+s}\ldots (is)^0 \ldots(s+L)^{g_{s+L}}).\]
It is easy to check that $\psi_2(\beta)\in C^2_{L,s,2}$ and $\psi_2(\phi_2(\alpha))=\alpha$. This completes the proof.\qed

\begin{exa} Let $s=3$, $L=103$, $k=103$ and $n=1286$. Let
\[\alpha=(10^1,11^3,20^7,28^2,31^7,46^9,52^3,65^4)\]
which is in $C^2_{103,3,2}(1286)$. It is clear that $j=11$. Applying $\phi_2$ to $\alpha$, we see that
\[\phi_2(\alpha)=(3^1,10^1,20^7,28^2,30^1,31^7,46^9,52^3,65^4)\]
which is in $F^2_{103,3,103}(1286)$. Applying $\psi_2$ to $\phi_2(\alpha)$, we have $i=10$ and $\psi_2(\phi_2(\alpha))=\alpha$.
\end{exa}

\begin{lem}\label{lem-inj-3}
There is an explicit injection $\phi_3: C^3_{L,s,2}(n)\to F^3_{L,s,k}(n)$.
\end{lem}

\pf Given $\alpha=((s+1)^{f_{s+1}}\ldots (s+L)^{f_{s+L}})\in C^3_{L,s,2}(n)$, by definition we see $f_k=0$, and for any $a\geq 2$ we have $f_{as}=0$. Moreover, $f_j\leq s-1$ for any $2s^2+5s-1\geq j\geq s+1$. We claim that there exists $i\geq 2s^2+5s+1$ such that $f_i\geq 1$. Otherwise, we see that $n=|\alpha|\leq (s-1)(s+1)+(s-1)(s+2)+\cdots+(s-1)(2s^2+5s-1)=2s^5+8s^4+s^3-14s^2+3s$, which is contradict to $n\ge 2s^5+8s^4+s^3-14s^2+3s+1$. Hence our claim has been verified.

From the above claim, we may set $j=\min\{i\colon i\geq 2s^2+5s+1\}$ and $j=cs+d$, where $1\leq d\leq s-1$. From $j\geq 2s^2+5s+1$ we see that $c\geq 2s+5$. Moreover, it is trivial to check that
\begin{equation}\label{equ-te-1-12}
j=cs+d=s+(s+1)(s+d)+s(c-s-d-2).
\end{equation}
Notice that $c-s-d-2\geq 2s+5-s-(s-1)-2=4$. It is well known that  $c-s-d-2$ can be uniquely written  as $2x+3y$, where $0\leq y\leq 1$. Now we may define $\phi_3(\alpha)$ as follows.
\begin{eqnarray}\label{equ-phi3}
\phi_3(\alpha)&=&(s^{g_s}\ldots (s+L)^{g_{s+L}})\nonumber\\
&=&(s^{1}\ldots  (s+d)^{f_{s+d}+s+1}\ldots (2s)^x \ldots (3s)^y\ldots j^{f_j-1}\ldots (s+L)^{f_{s+L}}).
\end{eqnarray}
From $c-s-d-2\geq 4$ we see that $g_{2s}+g_{3s}=x+y\geq 2$. Moreover, $k\geq L> 3s$ yields that $g_k=f_k=0$. Furthermore, we may calculate $|\phi_3(\alpha)|$ as follows.
\begin{eqnarray}\label{equ-phi3-weight}
|\phi_3(\alpha)|&=&|\alpha|+s+(s+1)(s+d)+x\cdot 2s+y\cdot 3s -j\nonumber\\
&=&n+s+(s+1)(s+d)+s(2x+3y)-j\nonumber\\
&=&n+s+(s+1)(s+d)+s(c-s-d-2)-j\nonumber\\
&=&n.
\end{eqnarray}
The last equation follows from \eqref{equ-te-1-12}. Hence  $\phi_3(\alpha)\in F^3_{L,s,k}(n)$.

It remains to show that  $\phi_3$ is an injection. Let
$$I^3_{L,s,k}(n)=\{\phi_3(\alpha)\colon \alpha\in C^3_{L,s,2}(n)\}$$
 be the image set of $\phi_3$, which has been shown to be a subset of $F^3_{L,s,k}(n)$. We wish to construct a map $\psi_3\colon I^3_{L,s,k}(n)\rightarrow C^3_{L,s,2}(n)$ such that for any $\alpha\in C^3_{L,s,2}(n)$,
 \[\psi_3(\phi_3(\alpha))=\alpha.\]
 Let $\beta=(s^{g_s}\ldots (s+L)^{g_{s+L}})\in I^3_{L,s,k}(n)$, that is, there exists $\alpha\in C^3_{L,s,2}(n)$ such that $\phi_3(\alpha)=\beta$. From the definition of $F^3_{L,s,k}(n)$, we see that  $g_{2s}+g_{3s}\geq 2$. By \eqref{equ-phi3} and the fact $f_i\leq s-1$ for all $s+1\leq i\leq 2s^2+5s-1$, we see that there exists a unique $s+1\leq i\leq 2s-1$ such that $g_i\geq s+1$. Moreover, by \eqref{equ-phi3-weight}, we see that
 $$s+2sx+3sy+(s+d)(s+1)=j.$$
  Thus
  $$s+2sg_{2s}+3sg_{3s}+i(s+1)=j\neq k.$$
   Hence we may  define
 \[\psi_3(\beta)=((s+1)^{g_{s+1}}\ldots i^{g_{i}-s-1}\ldots 2s^0\ldots 3s^0\ldots w^{g_w+1} \ldots(s+L)^{g_{s+L}}),\]
 where $w=s+2sg_{2s}+3sg_{3s}+i(s+1)$.
It is easy to check that $\psi_3(\beta)\in C^3_{L,s,2}(n)$ and $\psi_3(\phi_3(\alpha))=\alpha$. This completes the proof.\qed

\begin{exa}
For example, let $s=3$, $L=105$, $k=105$ and $n=1057$. Let
\[\alpha=(4^2,7^2,11^2,13^2,16^2,19^2,32^2,55,58^3,61^4,76^5)\]
be a partition in $C^3_{105,3,2}(1057)$. It is easy to see that $j=55$. Hence $c=18$ and $d=1$ and $c-s-d-2=12=2*6$. So $x=6$ and $y=0$. Applying $\phi_3$ on $\alpha$,
\[\phi_3(\alpha)=(3,4^6,6^6,7^2,11^2,13^2,16^2,19^2,32^2,58^3,61^4,76^5).\]
It is trivial to check that $\phi_3(\alpha)\in F^3_{105,3,105}(1057)$. Applying $\psi_3$ to $\phi_3(\alpha)$ we recover $\alpha$.
\end{exa}

\begin{lem}\label{lem-inj-4}
There is an explicit injection $\phi_4 : C^4_{L,s,2}(n)\to F^4_{L,s,2}(n)$.
\end{lem}

\pf Given $\alpha=((s+1)^{f_{s+1}}\ldots (s+L)^{f_{s+L}})\in C^4_{L,s,2}(n)$, by definition we see $f_k\geq 2$. Recall that $k=rs+t$, where $0\leq t\leq s-1$ and $r\geq 2s^2+5s\geq 7$. We may define $\phi_4(\alpha)$ as follows.
\begin{equation}\label{equ-phi4}
\phi_4(\alpha)=(s^{g_s}\ldots (s+L)^{g_{s+L}})
=(s^{f_k(r-2)}\ldots  (2s+t)^{f_{2s+t}+f_k}\ldots k^0 \ldots (s+L)^{f_{s+L}}).
\end{equation}
It is clear that $g_s=f_k(r-2)\geq 2(r-2)$. Moreover, $2s+t<7s+t\leq rs+t= k$ and $$|\phi_4(\alpha)|=|\alpha|+sf_k(r-2)+(2s+t)f_k-kf_k=n.$$
Hence  $\phi_4(\alpha)\in F^4_{L,s,k}(n)$.

It remains to show that $\phi_4$ is an injection. Let
$$I^4_{L,s,k}(n)=\{\phi_4(\alpha)\colon \alpha\in C^4_{L,s,2}(n)\}$$
 be the image set of $\phi_4$, which has been shown to be a subset of $F^4_{L,s,k}(n)$. We wish to construct a map $\psi_4\colon I^4_{L,s,k}(n)\rightarrow C^4_{L,s,2}(n)$ such that for any $\alpha\in C^4_{L,s,2}(n)$,
 \[\psi_4(\phi_4(\alpha))=\alpha.\]
 Let $\beta=(s^{g_s}\ldots (s+L)^{g_{s+L}})\in I^4_{L,s,k}(n)$, that is, there exists $\alpha\in C^4_{L,s,2}(n)$ such that $\phi_4(\alpha)=\beta$. From \eqref{equ-phi4}, the construction of $\phi_4$, we see that $g_s$ is a multiple of $(r-2)$. Moreover, $g_{2s+t}\geq g_s/(r-2)$. We may define $\psi_4$ as follows.
 \[\psi_4(\beta)=((s+1)^{g_{s+1}}\ldots (2s+t)^{g_{2s+t}-g_{s}/(r-2)}\ldots k^{g_s/(r-2)} \ldots(s+L)^{g_{s+L}}).\]
It is easy to check that $\psi_4(\beta)\in C^4_{L,s,2}(n)$ and $\psi_4(\phi_4(\alpha))=\alpha$. This completes the proof.\qed
\begin{exa}
For $s=3$, $L=108$, $k=109$ and $n=1138$, set
\[\alpha=(4^6,7^5,12^4,18^3,25^3,42^5,73^5,109^3).\]
Then $k=36*3+1$, so $r=36$ and $t=1$. Applying $\phi_4$ to $\alpha$, we derive that
\[\phi_4(\alpha)=(3^{102},4^6,7^8,12^4,18^3,25^3,42^5,73^5).\]
It is trivial to check that $\phi_4(\alpha)\in F^4_{108,3,109}(1138)$. Applying $\psi_4$ to $\phi_4(\alpha)$ we recover $\alpha$.
\end{exa}

\begin{lem}\label{lem-inj-5}
There is an explicit injection $\phi_5:C^5_{L,s,2}(n)\to F^5_{L,s,k}(n)$.
\end{lem}

\pf Given $\alpha=((s+1)^{f_{s+1}}\ldots (s+L)^{f_{s+L}})\in C^5_{L,s,2}(n)$, by definition we see $f_k=1$. Recall that $k=rs+t$, where $0\leq t\leq s-1$ and $r\geq 2s^2+5s\geq 7$. When $t\neq 0$,  define $\phi_5(\alpha)$ as follows.
\begin{equation}\label{equ-phi5-1}
\phi_5(\alpha)=(s^{g_s}\ldots (s+L)^{g_{s+L}})
=(s^{r-4}\ldots (2s)^{f_{2s}+1}\ldots (2s+t)^{f_{2s+t}+1}\ldots k^0 \ldots (s+L)^{f_{s+L}}).
\end{equation}
And when $t=0$, we set $\phi_5(\alpha)$ as given below.
\begin{equation}\label{equ-phi5-2}
\phi_5(\alpha)=(s^{g_s}\ldots (s+L)^{g_{s+L}})
=(s^{r-4}\ldots (2s)^{f_{2s}+2}\ldots k^0 \ldots (s+L)^{f_{s+L}}).
\end{equation}
In either case, we see that $g_s=r-4$ and $g_{2s}\geq 1$. Moreover, it is trivial to check that $|\phi_5(\alpha)|=n$. This yields $\phi_5(\alpha)\in F^5_{L,s,k}(n)$.

It remains to show that $\phi_5$ is an injection. Let
$$I^5_{L,s,k}(n)=\{\phi_5(\alpha)\colon \alpha\in C^5_{L,s,2}(n)\}$$
 be the image set of $\phi_5$, which has been shown to be a subset of $F^5_{L,s,k}(n)$. We wish to construct a map $\psi_5\colon I^5_{L,s,k}(n)\rightarrow C^5_{L,s,2}(n)$ such that for any $\alpha\in C^5_{L,s,2}(n)$,
 \[\psi_5(\phi_5(\alpha))=\alpha.\]
 Let $\beta=(s^{g_s}\ldots (s+L)^{g_{s+L}})\in I^5_{L,s,k}(n)$, that is, there exists $\alpha\in C^5_{L,s,2}(n)$ such that $\phi_5(\alpha)=\beta$. When $t\neq 0$, from \eqref{equ-phi5-1} we see $g_s=r-4$, $g_{2s}\ge 1$ and $g_{2s+t}\geq 1$. We may define $\psi_5$ as follows.
 \[\psi_5(\beta)=((s+1)^{g_{s+1}}\ldots (2s)^{g_{2s}-1}\ldots (2s+t)^{g_{2s+t}-1}\ldots k^{1} \ldots(s+L)^{g_{s+L}}).\]
 When $t=0$, from \eqref{equ-phi5-2}, we have $g_s=r-4$ and $g_{2s}\ge 2$. Hence the map $\psi_5$ is defined as follows.
  \[\psi_5(\beta)=((s+1)^{g_{s+1}}\ldots (2s)^{g_{2s}-2}\ldots  k^{1} \ldots(s+L)^{g_{s+L}}).\]
It is easy to check that in either case $\psi_5(\beta)\in C^5_{L,s,2}(n)$ and $\psi_5(\phi_5(\alpha))=\alpha$. This completes the proof.\qed

\begin{exa}
For $s=3$, $L=103$, $k=105$ and $n=1217$, set
\[\alpha=(6^2,9^5,12^8,17^4,35^6,42^5,73^5,105^1,106^1).\]
Then $k=35*3$, so $r=35$ and $t=0$. Applying $\phi_5$ to $\alpha$, we derive that
\[\phi_5(\alpha)=(3^{31},6^4,9^5,12^8,17^4,35^6,42^5,73^5,106^1).\]
It is trivial to check that $\phi_5(\alpha)\in F^5_{103,3,105}(1217)$. Applying $\psi_5$ to $\phi_5(\alpha)$ we recover $\alpha$.
\end{exa}

We are now in a position to prove Theorem \ref{thm-main-inf}.

{\noindent \it Proof of Theorem \ref{thm-main-inf}.} Given integer $s\geq 1$, $L\geq 2s^3+5s^2+1$, $s+L\ge k\ge L$ and $n\ge 2s^5+8s^4+s^3-14s^2+3s+1$,
for any $\alpha\in C_{L,s,2}(n)$,  we define
\begin{equation}
\phi(\alpha)=\phi_i(\alpha)\;\; \text{if }\;\; \alpha\in C^i_{L,s,2}(n)\;\; \text{for}\;\; i=1,\ldots, 5.
\end{equation}
From Lemmas~\ref{lem-inj-1}-\ref{lem-inj-5},
 we deduce that $\phi(\alpha)\in F_{L,s,k}(n)$ and $\phi$ is an injection. This completes the proof.\qed

We show that   Theorem \ref{thm-main} is  a consequence of
Theorems~\ref{thm-asy} and \ref{thm-main-inf}.

{\noindent \it Proof of Theorem \ref{thm-main}.} From Theorem \ref{thm-asy}, for any positive integer $s$ and $L$, there exists an integer  $M_{L,s}$ such that for $n\geq M_{L,s}$,
\begin{equation}\label{eq-f-s-k-l}
f_{L,s,k}(n)\geq c_{L,s,2}(n).
\end{equation}
Moreover, by Theorem \ref{thm-main-inf}, for $L\geq 2s^3+5s^2+1$ and $n\geq 2s^5+8s^4+s^3-14s^2+3s+1$, \eqref{eq-f-s-k-l} also holds. Hence, if  we set
\[M=\max\{M_{3,s},M_{4,s},\ldots,M_{2s^3+5s^2,s},2s^5+8s^4+s^3-14s^2+3s+1\},\]
then $M$ only depends on $s$, and \eqref{eq-f-s-k-l} holds for all $n\geq M$.\qed

\section{Proof of Theorem \ref{thm-main0-2}}

Define the sequences $(a_n)$, $(b_n)$ and  $(c_n)$ by
\begin{align}
\sum_{n=0}^\infty a_nq^n&:=\frac{1-q}{(q^s;q)_{L+1}},\label{equ-def-an}\\
\sum_{n=0}^\infty b_nq^n&:=\frac{1}{(1-q^s)(q^{s+2};q)_{L-1}},\label{equ-def-bn}\\
\sum_{n=0}^\infty c_nq^n&:=\frac{1}{(q^{s+1};q)_L}.\label{equ-def-cn}
\end{align}
We have the following result, which will play a curial role in the proof of Theorem \ref{thm-main0-2}.

\begin{lem}\label{lem-1-asy}
For any $s\geq 1$ and $L\geq 3$,
\begin{equation}
a_n\sim \frac{n^{L-1}}{(L-1)!s(s+1)\cdots (s+L)}.
\end{equation}
\end{lem}
\pf Writing $1-q=(1-q^{s+1})-(q-q^{s+1})$ we see that
\begin{align}\label{equ-sum-an}
\frac{1-q}{(q^s;q)_{L+1}}
=\frac{1}{(1-q^s)(q^{s+2};q)_{L-1}}-\frac{q}{(q^{s+1};q)_L}.
\end{align}
It follows from \eqref{equ-def-an}-\eqref{equ-def-cn} that
\begin{equation}\label{equ-sum-an}
a_n=b_n-c_{n-1}\quad \textrm{for} \quad n\geq 1.
\end{equation}
 For $L\geq 3$, as $\gcd(s,s+2,\ldots,s+L)=1$ and $\gcd(s+1,\ldots,s+L)=1$,
by Theorem \ref{schur-thm}, we have
\begin{equation}\label{equ-sim-bn}
b_n\sim \frac{(s+1)n^{L-1}}{(L-1)!s(s+1)\cdots (s+L)},
\end{equation}
and
\begin{equation}\label{equ-sim-cn}
c_n\sim \frac{sn^{L-1}}{(L-1)!s(s+1)\cdots (s+L)}.
\end{equation}
Substituting
\eqref{equ-sim-bn} and \eqref{equ-sim-cn} into \eqref{equ-sum-an}, we deduce that
\begin{equation}
a_n=b_n-c_{n-1}\sim \frac{n^{L-1}}{(L-1)!s(s+1)\cdots (s+L)}.
\end{equation}
This completes the proof.\qed

We proceed to show Theorem \ref{thm-main0-2} with the aid of Lemma \ref{lem-1-asy}.

{\noindent \it Proof of Theorem \ref{thm-main0-2}.}
For fixed integer $k$, define
\begin{align}\label{def:d}
\sum_{n=0}^\infty d_{k}(n)q^n=\frac{1-q^{k}}{(q^s;q)_{L+1}}.
\end{align}
Comparing with \eqref{equ-def-an}  we have
\begin{equation}
\sum_{n=0}^\infty d_k(n)q^n
=(1+q+\cdots+q^{k-1})\sum_{n=0}^\infty a_nq^n.
\end{equation}
Thus,  for $n\geq k$,  Lemma \ref{lem-1-asy} implies that
\begin{align}
d_k(n)&=a_n+a_{n-1}+\cdots+a_{n-k+1}\nonumber\\
&\sim \frac{n^{L-1}+\cdots+(n-k+1)^{L-1}}{(L-1)!s(s+1)\cdots (s+L)}\nonumber\\
&\sim \frac{kn^{L-1}}{(L-1)!s(s+1)\cdots (s+L)}.\label{equ-dkn-sim-in}
\end{align}
For fixed integer $r\geq 0$, set
\begin{align}\label{def:e}
\sum_{n=0}^\infty e_{k,r}(n)q^n:=\frac{q^r(1-q^{k})}{(q^s;q)_{L+1}}.
\end{align}
Then, for $n\ge r$,
\begin{equation}\label{equ-temp-e-d}
e_{k,r}(n)=d_k(n-r).
\end{equation}
By \eqref{equ-dkn-sim-in}, we deduce that
\begin{equation}\label{equ-dkn-sim-in-en}
e_{k,r}(n)\sim \frac{kn^{L-1}}{(L-1)!s(s+1)\cdots (s+L)}.
\end{equation}
Therefore, it follows from \eqref{def:H}, \eqref{def:d} and \eqref{def:e} that
\begin{equation}
H^*_{L,s,r,k_1,k_2}(q)=\sum_{n=0}^\infty \left(e_{k_1,r}(n)-d_{k_2}(n)\right)q^n.
\end{equation}
From \eqref{equ-dkn-sim-in} and \eqref{equ-dkn-sim-in-en}, we derive that
\[e_{k_1,r}(n)-d_{k_2}(n)\sim \frac{(k_1-k_2)n^{L-1}}{(L-1)!s(s+1)\cdots (s+L)}.\]
Since $k_1>k_2$, there exists $M$ such that
$e_{k_1,r}(n)-d_{k_2}(n)>0$ for $n>M$.
Thus  $H^*_{L,s,r,k_1,k_2}(q)$ is eventually positive. This completes the proof.\qed

\begin{rem} For some special cases it would be interesting to determine the smallest $M$ in
Theorem~\ref{thm-main}, see also Conjecture 5.3 in \cite{Ber-Unc-19}.
\end{rem}

\noindent{\bf Acknowledgments.}
This work was done during the second author's visit to the Harbin Institute of Technology (HIT)
in the summer of 2019.
The first author was supported by the National
Natural Science Foundation of China (No. 11801119).
 The second author would like to thank
 Institute for  Advanced Study in Mathematics of HIT for the hospitality.

\end{document}